\documentclass[graybox]{svmult}
\usepackage{amsmath}
\usepackage{eucal}
\usepackage[active]{srcltx}
\usepackage{amssymb}
 \usepackage[usenames,dvipsnames]{pstricks}
 \usepackage{enumerate}
 \usepackage{epsfig}
\usepackage[ansinew]{inputenc} 

\renewcommand{\>}{\rangle}
\newcommand{\p}{\partial}
\newcommand{\pfrac}[2]{\genfrac{}{}{}{1}{#1}{#2}}

\newcommand{\at}[2]{\genfrac{}{}{0pt}{}{#1}{#2}}

\newcommand{\bb}[1]{{\mathbb #1}}
\newcommand{\mc}[1]{{\mathcal #1}}

\usepackage{tikz}

\newcommand{\Y}{\mathcal{Y}}

\usepackage{helvet}         
\usepackage{courier}        

\usepackage{graphicx}        
\usepackage{multicol}        
\definecolor{columbiablue}{rgb}{0.61, 0.87, 1.0}

\usepackage{tikz}
\usetikzlibrary{arrows,decorations.pathmorphing,backgrounds,positioning,fit,petri}
\usetikzlibrary{shapes}

\newcommand{\tclock}[5]{
\begin{pgflowlevelscope}{\pgftransformscale{#4}}
\begin{scope}[shift={(#1,#2)}]
\shadedraw [inner color=#3!7!gray, outer color=#3!90!black, 
    line width=0.2pt] (0,0) circle (0.5cm);
\foreach \x in {6,12,...,360} {\draw[line width=0.2pt] (\x:0.40cm) -- (\x:0.45cm);}
\foreach \y in {30,60,...,360} {\draw[line width=0.2pt] (\y:0.35cm) -- (\y:0.45cm);}
{\pgfsetarrowsstart{to}
\draw[line width=0.4pt] (0:0.29cm) -- (0.02,0);
\draw[line width=0.4pt]  (90:0.32cm)--(0,0.02);}
\filldraw[fill=black] (-0.055,0.55) rectangle (0.055,0.6);
\filldraw[fill=black] (-0.015,0.51) rectangle (0.015,0.55);
\draw [line width=0.2pt](0,0.61) circle (0.11cm);
\draw [line width=0.2pt](0,0) circle (0.5cm);
\draw [line width=0.2pt](0,0) circle (0.02cm);
\draw [white,thick,domain=30:45] plot ({#5*0.6*cos(\x)}, {#5*0.6*sin(\x)});
\draw [white,thick,domain=20:55] plot ({#5*0.65*cos(\x)}, {#5*0.65*sin(\x)});
\draw [white,thick,domain=10:65] plot ({#5*0.7*cos(\x)}, {#5*0.7*sin(\x)});
\draw [white,thick,domain=135:150] plot ({#5*0.6*cos(\x)}, {#5*0.6*sin(\x)});
\draw [white,thick,domain=125:160] plot ({#5*0.65*cos(\x)}, {#5*0.65*sin(\x)});
\draw [white,thick,domain=170:115] plot ({#5*0.7*cos(\x)}, {#5*0.7*sin(\x)});
\end{scope}
\end{pgflowlevelscope}
}

\begin{document}

\title*{Equilibrium fluctuations for the slow boundary exclusion process}
\author{Tertuliano Franco, Patr\' icia Gon\c calves and Adriana Neumann}
\institute{Tertuliano Franco \at UFBA,
 Instituto de Matem\'atica, Campus de Ondina, Av. Adhemar de Barros, S/N. CEP 40170-110,
Salvador, Brazil.
\\ \email{tertu@ufba.br}
\and Patrícia Gonçalves \at  Center for Mathematical Analysis,  Geometry and Dynamical Systems,
Instituto Superior T\'ecnico, Universidade de Lisboa,
Av. Rovisco Pais, 1049-001 Lisboa, Portugal
\\ \email{patricia.goncalves@math.tecnico.ulisboa.pt}
\and Adriana Neumann \at UFRGS, Instituto de Matem\'atica, Campus do Vale, Av. Bento Gon\c calves, 9500. CEP 91509-900, Porto Alegre, Brazil.
\\ \email{aneumann@mat.ufrgs.br}}

\maketitle

\abstract{We prove that the equilibrium fluctuations of the symmetric simple exclusion process in contact with slow boundaries is given by an Ornstein-Uhlenbeck process with Dirichlet, Robin or Neumann boundary conditions depending on the range of the parameter that  rules the slowness of the boundaries.}

\section{Introduction}

The study of nonequilibrium behavior of interacting particle systems is one of the most challenging problems in the field and it has only been completely solved in very particular cases. The toy model for the study of a system in a nonequilibrium scenario is the symmetric simple exclusion process (SSEP) whose dynamics is rather simple to explain and it already captures many features of more complicated systems.  

 The dynamics of this model can be described as follows. We fix a scaling parameter $n$ and we  consider the SSEP evolving on the discrete space $\Sigma_n=\{1,\cdots, n-1\}$ to which we call the bulk. To each pair of bonds $\{x,x+1\}$ with $x=1,\cdots, n-2$ we associate a Poisson process $N_{x,x+1}(t)$ of rate $1$. Now we artificially add two end points at the bulk, namely, we add  the sites $x=0$ and $x=n$ and we  superpose the exclusion dynamics with a Glauber dynamics which has only effect at the boundary points of the bulk, namely at the sites $x=1$ and $x=n-1$. For that purpose,  we add extra Poisson processes at the bonds $\{0,1\}$ and $\{n-1,n\}$.  In each one of these bonds there are two Poisson processes:  $N_{0,1}(t)$ with parameter  $\alpha n^{-\theta}$, $N_{1,0}(t)$ with parameter $(1-\alpha)n^{-\theta}$,   $N_{n-1,n}(t)$ with parameter  $\beta n^{-\theta}$ and  $N_{n,n-1}(t)$ with parameter $(1-\beta)n^{-\theta}$. All the Poisson  processes are independent. Above $\alpha, \beta\in(0,1)$ and $\theta\geq 0$ is a parameter that rules the slowness of the boundary dynamics.  Below in the figure we colored the Poisson clocks associated to the bonds in the bulk in the blue color, while the Poisson clocks associated to the bonds at the boundary are colored in the  gray and pink colors, to emphasize that  
they have different rates.

Now that the clocks are fixed we can explain the dynamics. For that purpose,  initially we place particles in the bulk according to some probability measure and we denote  this configuration of particles and holes by $\eta=(\eta(1),\cdots, \eta(n-1))$, so that for $x\in\Sigma_n$, $\eta(x)=1$ if there is a particle at the site $x$ and $\eta(x)=0$ if the site $x$ is empty. Now, if a clock rings for a bond  $\{x,x+1\}$ in the bulk, then we exchange the coordinates $x$ and $x+1$ of $\eta$, that is we exchange $\eta(x) $ with $\eta(x+1)$ at rate $1$. If the clock rings for the bond at the boundary as, for example, from the Poisson process $N_{0,1}(t)$ then a particle gets into the bulk through the site $1$ at rate $\alpha n^{-\theta}$ if and only if there is no particle at the site $1$, otherwise nothing happens. If the clock rings from the Poisson process $N_{1,0}(t)$ and there is a particle at the site $1$, then it exits the bulk from the site $n-1$ at rate $(1-\alpha)n^{-\theta}$. 
 Note that the higher the value of $\theta$ the slower is the dynamics at the boundaries. For a display of the description above, see the figure below.
 
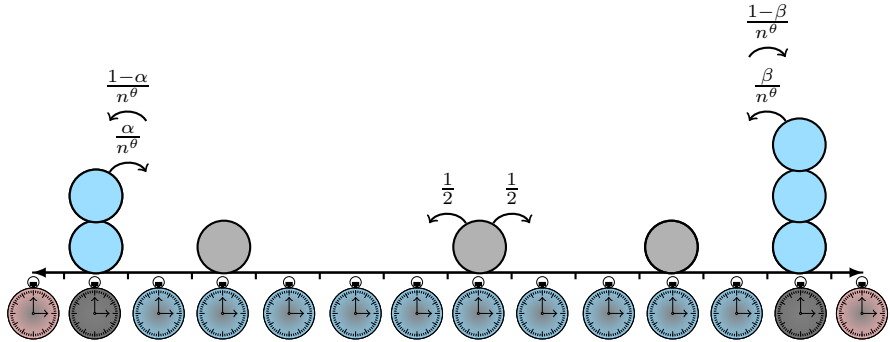
\begin{center}
 \begin{tikzpicture}[thick, scale=0.85][h!]
 \draw[latex-] (-6.5,0) -- (6.5,0) ;
\draw[-latex] (-6.5,0) -- (6.5,0) ;
\foreach \x in  {-6,-5,-4,-3,-2,-1,0,1,2,3,4,5,6}
\draw[shift={(\x,0)},color=black] (0pt,0pt) -- (0pt,-3pt) node[below] 
{};
 \node[fill=black!30!,shape=circle,draw=black,minimum size=0.7cm] (A) at (0.5,0.4) {};
       \node[shape=circle,minimum size=0.7cm] (N) at (-4.5,1.2) {};
       \node[shape=circle,minimum size=0.7cm] (P) at (-4.5,2.0) {};
    \node[fill=black!30!,shape=circle,draw=black,minimum size=0.7cm] (C) at (-3.5,0.4) {};
   \node[fill=black!30!,shape=circle,draw=black,minimum size=0.7cm] (E) at (3.5,0.4) {};
    \node[shape=circle,draw=black,minimum size=0.7cm] (E) at (3.5,0.4) {};
        \node[shape=circle,draw=black,minimum size=0.7cm] (L) at (-5.5,0.4) {};
         \node[fill=columbiablue,shape=circle,draw=black,minimum size=0.7cm] (L) at (-5.5,0.4) {};
          \node[shape=circle,draw=black,minimum size=0.7cm] (M) at (-5.5,1.2) {};
          \node[shape=circle,minimum size=0.7cm] (Q) at (-5.5,2.0) {};
           \node[fill=columbiablue,shape=circle,draw=black,minimum size=0.7cm] (M) at (-5.5,1.2) {};
           
             \node[shape=circle,draw=black,minimum size=0.7cm] (R) at (5.5,0.4) {};
          \node[shape=circle,minimum size=0.7cm] (S) at (5.5,1.2) {};
           \node[fill=columbiablue,shape=circle,draw=black,minimum size=0.7cm] (R) at (5.5,0.4) {};
            \node[fill=columbiablue,shape=circle,draw=black,minimum size=0.7cm] (S) at (5.5,1.2) {};
              \node[shape=circle,minimum size=0.7cm] (T) at (5.5,2.0) {};
               \node[shape=circle,minimum size=0.7cm] (U) at (4.5,3.6) {};
                 \node[shape=circle,minimum size=0.7cm] (V) at (5.5,3.0) {};
               
                  \node[shape=circle,minimum size=0.7cm] (W) at (4.5,3.0) {};
               
           \node[fill=columbiablue,shape=circle,draw=black,minimum size=0.7cm] (T) at (5.5,2.0) {};
              \node[shape=circle,minimum size=0.7cm] (U) at (4.5,2.0) {};
              \path [->] (T) edge[bend right =60] node[above] {$\frac{\beta}{n^\theta}$} (U);
              
                \path [->] (W) edge[bend left=60] node[above] {$\frac{1-\beta}{n^\theta}$} (V);
              
               \path [->] (P) edge[bend right=60] node[above] {${\frac{1-\alpha}{n^\theta}}$} (Q);
               \path [->] (M) edge[bend left=60] node[above] {$\frac{\alpha}{n^\theta}$} (N);
              
    \node[shape=circle,minimum size=0.7cm] (K) at (-0.5,0.4) {};
        \node[shape=circle,minimum size=0.7cm] (G) at (1.5,0.4) {};
  \path [->] (A) edge[bend right=60,draw=black] node[above] {$\frac 12$} (K);
    \path [->] (A) edge[bend left=60] node[above] {$\frac 12$} (G);
\tclock{0.6}{-0.8}{columbiablue}{0.8}{1}
\tclock{4.4}{-0.8}{columbiablue}{0.8}{0}
\tclock{-4.4}{-0.8}{columbiablue}{0.8}{0}
\tclock{-6.9}{-0.8}{gray}{0.8}{0}
\tclock{-8.1}{-0.8}{pink}{0.8}{0}
\tclock{-5.65}{-0.8}{columbiablue}{0.8}{0}
\tclock{-3.1}{-0.8}{columbiablue}{0.8}{0}
\tclock{-1.8}{-0.8}{columbiablue}{0.8}{0}
\tclock{-0.6}{-0.8}{columbiablue}{0.8}{0}
\tclock{3.15}{-0.8}{columbiablue}{0.8}{0}
\tclock{1.85}{-0.8}{columbiablue}{0.8}{0}
\tclock{5.65}{-0.8}{columbiablue}{0.8}{0}
\tclock{6.9}{-0.8}{gray}{0.8}{0}
\tclock{8.1}{-0.8}{pink}{0.8}{0}
 \end{tikzpicture}
 \end{center}
 
 \quad 
  
  \quad

 The dynamics just described is Markovian and can be completely characterized in terms of its infinitesimal generator given below in \eqref{ln}. We note that the space state of this Markov process is $\Omega_n:=\{0,1\}^{\Sigma_n}$. Observe that the bulk dynamics preserves the number of particles and our interest is to describe the space-time evolution of this conserved quantity as a solution of some partial differential equation called the \textit{hydrodynamic equation}.

 Note that for the choice $\alpha=\beta=\rho$ a simple computation shows that the Bernoulli product measure of parameter $\rho$ given by:
 $\nu^n_{\rho}(\eta \in\Omega_n:\eta(x)=1)=\rho$ is invariant under  the dynamics. For this choice of the parameters the boundary reservoirs have the same intensity and we do not see any induced current on the system. Nevertheless, in the case $\alpha\neq \beta$, let us say for example $\alpha<\beta$, there is a tendency to have more particles entering into the bulk from the right reservoir and leaving the system from the left reservoir. This is a current which is induced by the difference of the  density at the boundary reservoirs. Note that in the bulk the dynamics is symmetric. In the case  $\alpha\neq \beta$, since we have a finite state Markov process, there is only one stationary measure that we denote by $\mu_n^{ss}$ which is no longer a product measure as in the case $\alpha=\beta$.  By using the matrix ansazt method developed by \cite{BEvans,derrida,derrida2} and references therein, it is possible to obtain information about this measure and an important problem is to analyze the behavior of the system starting from this non-equilibrium stationary state. 
 
 We note that the hydrodynamic limit of this model was studied in \cite{bmns} and the hydrodynamic equations consist in the heat equation with different types of boundary conditions depending on the range of the parameter $\theta$. More precisely, for $0\leq \theta<1$ the heat equation has Dirichlet boundary conditions which fix the value of the density profile at the points $0$ and $1$ to be $\alpha$ and $\beta$, respectively.  In this case we do not see any difference at the macroscopic level with respect to the case $\theta=0$. Nevertheless, for $\theta=1$ the boundary dynamics is slowed enough in such a way that macroscopically the Dirichlet boundary conditions are replaced by a type of Robin boundary conditions. These Robin boundary conditions state that the rate at which particles are injected into the system through the boundary points, is given by the difference of the density at the bulk and the boundary. Finally for $\theta>1$, the boundaries are sufficiently slowed so that the Robin boundary conditions are replaced by Neumann  boundary conditions  stating that macroscopically there is no flux of particles from the boundary reservoirs.

We emphasize that there are many similar models to the one studied in these notes which we summarize as follows. In \cite{mariaeulalia1, me1,mariaeulalia2}, the authors consider a model where  removal of particles can only occur at an interval around the left boundary and the entrance of particles is allowed only at an interval around the right boundary.  
 Their model presents a current exchange between the two reservoirs and shows some similarities with our model for the choice $\theta=1$. Another case already studied in the literature (see \cite{f,LMO}) is when the boundary is not slowed, that corresponds  to our model for the choice $\theta=0$. As mentioned above, the hydrodynamic equation of this model has Dirichlet boundary conditions, see \cite{f} or the equation \eqref{nobound}. A similar model, whose hydrodynamic equation has both Dirichlet boundary conditions and Neumann boundary conditions, was studied in \cite{dfp}. The main difference, at the macroscopic level, is that the end points of the boundary conditions vary with time. The microscopic dynamics there is given by the
SSEP  evolving on $\bb Z$ with
additional births and deaths restricted to a subset of configurations where there is a leftmost hole
and a rightmost particle. In this situation,  at a fixed rate $ j$  birth of particles occur at the position of the leftmost
hole and at the same rate, independently, the rightmost particle dies.  
Another model which has a current is considered in \cite{b}. The dynamics evolves on the discrete torus $\bb Z/n\bb Z$ without reservoirs, but has a surprising phenomenon: a ``battery effect''. This effect produces a current of particles through the system and is due to a single abnormal bond, where the rates to cross from left to right and from right to left are different. Finally, another model which has similarities with the model we consider in these notes is the SSEP with a slow bond, which was studied  in \cite{fgn,fgn3,fgn2}. The dynamics evolves on the  discrete torus $\bb Z/n\bb Z$, and particles exchange positions between nearest neighbor bonds at rate 1, except  at one particular bond,  where the exchange occurs at rate $n^{-\beta}$. In this case   $\beta>0$ is a parameter that rules the slowness of the bond and for that reason  the bond is called the \emph{slow bond}.
The similarity between the slow bond model and the slow boundary model considered in these notes is that  if we ``open'' the discrete torus exaclty at the position of the slow bond, then the slow bond rives rise to a  slow boundary. In \cite{fgn,fgn2}  different hydrodynamic behaviors were obtained, depending on the range of the parameter $\beta$, more precisely,  the hydrodynamic equation  is, in all cases, the heat equation but the boundary conditions vary with the value of $\beta$, exhibiting three different regimes as for the slow boundary model, see \cite{bmns}.

 Our interest in these notes is to go further the hydrodynamical behavior in order to analyze the fluctuations around the hydrodynamical profile. To accomplish this, we restrict ourselves to the  case  $\alpha=\beta=\rho$ and starting from the stationary measure $\nu_\rho^n$ defined above. Our result states that the fluctuations starting from $\nu_\rho^n$ are given by an Ornstein-Uhlenbeck  process  solution of 
\begin{equation*}
d\mathcal{Y}_t=\Delta_\theta \mathcal{Y}_tdt+\sqrt{2\chi(\rho)t}\,\nabla_\theta \, d\mc{W}_t\,,
\end{equation*}
where $\chi(\rho)$ is the variance of $\eta(x)$ with respect to $\nu_\rho^n$, $\mc{W}_t$ is a space-time white noise of unit variance and $\Delta_\theta$ and $\nabla_\theta$ are, respectively,  the Laplacian and derivative operators defined on a space of test functions with different types of boundary conditions depending on the value of $\theta$. 
We note that the case $\theta=0$ was studied in \cite{LMO}  and the case  $\theta=1$ was studied in \cite{fgn4}. In those articles, the  nonequilibrium fluctuations were obtained starting from general initial measures, which include the equilibrium case $\nu^n_\rho$ treated here. We note however, that the case $\theta\neq 1$ is quite difficult to attack at the nonequilibrium scenario since we need to establish a local replacement  (see Lemma \ref{rep_lemma}) in order to close the martingale problem, which we can only prove starting the system from the  equilibrium state.  In a future work, we will dedicate to extending this result to the nonequilibrium situation as, for example, starting the system  from the 
steady state when $\alpha\neq \beta$.

Here follows an outline of these notes. In Section \ref{s2} we give the definition of the model, we recall from \cite{bmns} the hydrodynamic limit and we state our main result, namely, Theorem \ref{OU_limit}. In Section  \ref{sec_proof} we characterize the limit process by means of a martingale problem. Tightness is proved in Section \ref{s6} and in Section \ref{sec_rl} we prove the Replacement Lemma which is the most technical part of these notes.

\section{Statement of results}\label{s2}

\subsection{The model}
For $n\geq{1}$, we denote by $\Sigma_n$ the set $\{1,\cdots,n-1\}$ to which we call the bulk. The symmetric simple exclusion process with slow boundaries is a Markov process $\{\eta_t:\,t\geq{0}\}$ with state space $\Omega_n:=\{0,1\}^{\Sigma_n}$. The slowness of the boundaries is ruled by a parameter that we denote by $\theta\geq 0$. If $\eta$ is a configuration of the state space $\Omega$, then for $x\in\Sigma_n$, the random variable  $\eta(x)$ can take only two values, namely $0$ or $1$. {If $\eta(x)=0$, it means that the site $x$ is vacant, while $\eta(x)=1$ means that the site $x$ is occupied.} The dynamics of this model can be described as follows. In the bulk particles move according to continuous time random walks, but whenever a particle wants to jump to an occupied site, the jump is suppressed. At the left boundary, particles can be created (resp. removed) at rate $\alpha n^{-\theta}$ (resp. $(1-\alpha) n^{-\theta}$). At the right boundary, particles can be created (resp. removed) at rate $\beta n^{-\theta}$ (resp. $(1-\beta) n^{-\theta}$). 

Fix now a finite time horizon $T$. 
The Markov process $\{\eta_t(x): x\in\Sigma_n ; t\in[0,T]\}$ can be characterized in terms of its infinitesimal generator that we denote by $\mc L_{n}^\theta$ and is defined as follows.  For a  function $f:\Omega_n\rightarrow \bb{R}$, we have that 
\begin{equation}\label{ln}
\begin{split}
(\mc L_{n}^{\theta}f)(\eta)\;=\;
& \Big[\frac{\alpha}{n^\theta}(1-\eta(1))+\frac{(1-\alpha)}{n^\theta}\eta(1)\Big]\Big(f( \eta^{1})-f(\eta)\Big)\\
+& \Big[\frac{\beta}{n^\theta}(1-\eta(n-1))+\frac{(1-\beta)}{n^\theta}\eta(n-1)\Big]\Big(f(\eta^{n-1})-f(\eta)\Big)\\
+& \sum_{x=1}^{n-2}\Big(f(\sigma^{x,x+1}\eta)-f(\eta)\Big) \,,
\end{split}
\end{equation}
where $\sigma^{x,x+1}\eta$ is the configuration obtained from $\eta$ by exchanging the occupation variables $\eta(x)$ and $\eta(x+1)$, that is,
\begin{equation}\label{sigma}
(\sigma^{x,x+1}\eta)(y)=\left\{\begin{array}{cl}
\eta(x+1),& \mbox{if}\,\,\, y=x\,,\\
\eta(x),& \mbox{if} \,\,\,y=x+1\,,\\
\eta(y),& \mbox{otherwise.}
\end{array}
\right.
\end{equation}
 and for $x=1,n-1$ $\eta^x$ is the configuration obtained from $\eta$ by flipping  the occupation  variable $\eta(x)$:
 \begin{equation}\label{split}
(\eta^x)(y)=\left\{\begin{array}{cl}
1-\eta(y),& \mbox{if}\,\,\, y=x\,,\\
\eta(y),& \mbox{otherwise.}
\end{array}
\right.
\end{equation}

Let  $\mc D([0,T],\Omega_n)$ be  the space of trajectories which are right continuous and  with left limits, taking values in $\Omega_n$.
Denote by $\bb P_{\mu_{n}}^{\theta,n} $ the probability on  $\mc D([0,T],\Omega_n)$ induced by the Markov process with generator $n^2\mc L_n^\theta$ and the initial measure $\mu_n$ and denote by $\bb E_{\mu_n}^{\theta,n}$ the expectation with respect to $\bb P_{\mu_n}^{\theta,n}$. 
%

\subsection{Stationary measures}\label{ss2.2}
The stationary measure $\mu_n^{ss}$ for this model when $\alpha=\beta=\rho\in(0,1)$ is the Bernoulli product measure given by $$\nu^n_\rho\Big(\eta\in \Omega_n:\eta(x)=1\Big)=\rho\,.$$ But in the general case, where  $\alpha\neq \beta$, the stationary measure  $\mu_n^{ss}$ does not have independent marginals,  see  \cite{derrida}.
What we can say about the
 stationary behavior of this model is  that the density of particles has a behavior very close to a linear profile, which  depends on the range of $\theta$ in the sense of the following definition:
\begin{definition}\label{eq3} Let $\gamma:[0,1]\to[0,1]$ be a measurable profile.
A sequence $\{\mu_n\}_{n\in \bb N}$ is said to be \textit{associated} to $\gamma$  if,  for any $ \delta >0 $ and any continuous function $ f: [0,1]\to\bb R $ the following limit holds:
\begin{equation*}
\lim_{n\to\infty}
\mu_{n} \Bigg(\, \eta:\, \Big| \frac {1}{n} \sum_{x = 1}^{n\!-\!1} f(\pfrac{x}{n})\, \eta(x)
- \int f(u)\, \gamma(u)\, du \Big| > \delta\,\Bigg) \;=\; 0\,.
\end{equation*}
\end{definition}
For $\mu_{n}$ equal to  the stationary measure $\mu^{ss}_n$, the limit above is called  the \emph{hydrostatic limit}. 

\begin{theorem}[Hydrostatic Limit, \cite{bmns}]
\quad

Let $\mu_n^{ss}$ be the stationary  probability measure in $\Omega_n$  wrt the Markov process with infinitesimal generator $n^2 \mc L_n^\theta$, defined in \eqref{ln}. The sequence $\{\mu_n^{ss}\}_{n\in\mathbb{N}}$ is associated (in the sense of  Definition \ref{eq3}) to the profile $\overline{\rho}: [0,1]\to\mathbb R$ given by
\begin{equation}\label{perfis_estacionarios}
\overline{\rho}(u)\;=\;
\begin{cases}
(\beta - \alpha)u+\alpha, & \mbox{if }\theta\in [0,1), \\
\frac{\beta - \alpha}{3}u + \alpha +\frac{\beta-\alpha}{3}, & \mbox{if }\theta = 1, \\
\frac{\beta + \alpha}{2}, & \mbox{if }\theta \in(1,\infty),
\end{cases}
\end{equation}
for all $u\in [0,1]$.
\end{theorem}
Another  feature that we can say about the stationary state of the model studied in this paper is that the profiles in \eqref{perfis_estacionarios} are very close to the mean of $\eta(x)$ taken with respect to the stationary measure $\mu_n^{ss}$. To state this result properly, we start by defining for an initial measure $\mu_n$ in $\Omega_n$, for $x\in \Sigma_n$ and for $t\geq 0$ the empirical mean given by 
\begin{equation}\label{mean}
\rho^n_t(x):=\bb E_{\mu_n}^{\theta,n}[\eta_t(x)]\,.
\end{equation} 
If in the expression above  $\mu_n=\mu_n^{ss}$, then
$\rho^n_t(x)$ does not depend on $t$, so that $\rho^n_t(x)=\rho^n(x)$. From \cite{bmns}, we have that $\rho^n(x)$ satisfies the following recurrence relations:
	\begin{equation*}
	\begin{cases}
	0 =[\rho^n(x+1)-\rho^n(x)]+[\rho^n(x-1)-\rho^n(x)],\quad \mbox{if } \;\;\;x\in\{2,\dots,n\!-\!2\},\\
	0 =[\rho^n(2)-\rho^n(1)]+n^{-\theta}[\alpha - \rho^n(1)], \\
	0 =n^{-\theta}[\beta - \rho^n(n\!-\!1)]+[\rho^n(n\!-\!2)-\rho^n(n\!-\!1)].
	\end{cases}
	\end{equation*}
	A simple computation shows that $\rho^n(x)$ is given by
	$\rho^n(x)=a_nx+b_n,$ for all  $ x\in \Sigma_n,$
	where 
	$a_n=\frac{\beta -\alpha}{2n^\theta+n\!-\!2} $ and $b_n=\alpha + a_n(n^\theta -1).$
	Moreover, we conclude that
	\begin{equation*}
	\lim_{n\to\infty}\Big(\max_{x\in \Sigma_n}\big|\rho^n(x)-\overline{\rho}(\pfrac{x}{n})\big|\Big)=0.
	\end{equation*}
	

\subsection{Hydrodynamic limit}

In \cite{bmns} it was  established the hydrodynamic limit of the model for any  $\theta\geq 0$. For completeness we recall that  result now. Fix a measurable density profile $ \rho_0: [0,1] \rightarrow [0,1 ]$ and for each  $n \in \bb N$, let $\mu_n$ be a probability measure on $\Omega_n$.

\begin{theorem}[Hydrodynamic Limit, \cite{bmns}]
\quad

 Suppose that the 
sequence $\{\mu_n\}_{n\in \bb N}$ is \textit{associated} to a profile $\rho_0(\cdot)$ in the sense of Definition \ref{eq3}. 
Then,  for each $ t \in [0,T] $, for any $ \delta >0 $ and any continuous function $ f:[0,1]\to\bb R $, 
\begin{equation*}
\lim_{ n \rightarrow +\infty }
\bb P_{\mu_{n}}^{\theta,n} \Bigg[\eta_{\cdot} : \Big\vert \frac{1}{n} \sum_{x =1}^{n-1}
f(\pfrac{x}{n})\, \eta_{tn^2}(x) - \int f(u)\, \rho(t,u)\, du\, \Big\vert
> \delta \Bigg] \;=\; 0,
\end{equation*}
 where  $\rho(t,\cdot)$ is the unique weak solution of the 
 heat equation 
 
 \begin{equation}\label{heat_eq}
\begin{cases}
\p_t \rho(t,u)= \p_u^2 \rho(t,u)\,, & \textrm{for } t>0\,,\, u\in (0,1),\\
\rho(0,u)=\rho_0(u)\,,& u\in [0,1].
\end{cases}
\end{equation}
with boundary conditions that depend on the range of $\theta$, which are given by:
  \begin{align}
 & \text{For } 
 \theta<1,  \;\p_u \rho(t,0) = \alpha  \text{ and } \p_u \rho(t,1) = \beta, \text{  for } t>0. \label{nobound}\\
 &  \text{For } \theta=1,\; \p_u \rho(t,0) = \rho(t,0)-\alpha \text{ and } \p_u \rho(t,1) = \beta-\rho(t,1), \text{ for } t>0.\label{Robin}\\
& \text{For }\theta>1, \;\p_u \rho(t,0) = \p_u \rho(t,1) =0, \text{ for }t>0.\label{Neumann} 
 \end{align}
\end{theorem}
\begin{remark}
We note that the profiles in \eqref{perfis_estacionarios} are stationary solutions of the heat equation with the corresponding boundary conditions  given above.
\end{remark}

\section{Density fluctuations}

\subsection{The space of test functions}
The space $ C^\infty([0,1])$ is the space of functions  $f:[0,1]\to \bb R$ such that $f$ is continuous in $[0,1]$ as well as   all its  derivatives. 
\begin{definition}\label{def1} Let $\mathcal S_{\theta}$ denote the set of functions $f\in C^\infty([0,1])$ such that for any $k\in\mathbb{N}\cup \{0\}$ it holds that 

\begin{enumerate}[(1)]

\item for $\theta<1$: 
$ \p_u^{2k} f(0)=\p_u^{2k} f(1)=0.$ 

\item for $\theta=1$:
$ \p_u^{2k+1} f(0)= \p_u^{2k}f(0)$ and $\p_u^{2k+1} f(1)=- \p_u^{2k}f(1).
$

\item for $\theta>1$:
$ \p_u^{2k+1} f(0)=\p_u^{2k+1} f(1)=0.$ 
\end{enumerate}
\end{definition}

\begin{definition}\label{def:laplacian_operator}

For $\theta\geq 0$, let $-\Delta_\theta$ be the positive operator, self-adjoint on $L^2{[0,1]}$, defined  on $f\in\mc S_\theta$ by
\begin{equation}\label{laplacian}
\Delta_\theta f(u)\;=\;\left\{\begin{array}{cl}
\partial_u^2 f(u)\,,& \mbox{if}\,\,\,u\in(0,1),\\
\partial_u^2 f(0^+)\,,& \mbox{if} \,\,\,u=0,\\
\partial_u^2 f(1^-)\,,& \mbox{if} \,\,\,u=1.\\
\end{array}
\right.
\end{equation}
Above, $\partial_u^2 f(a^\pm)$ denotes the side limits at the point $a$.  Analogously, let\break   $\nabla_\theta: \mc S_\theta\rightarrow C^\infty([0,1])$  be the operator given by 
\begin{equation}\label{nabla}
\nabla_\theta f(u)\;=\;\left\{\begin{array}{cl}
\partial_u f(u)\,,& \mbox{if}\,\,\,u\in(0,1),\\
\partial_u f(0^+)\,,& \mbox{if} \,\,\,u=0,\\
\partial_u f(1^-)\,,& \mbox{if} \,\,\,u=1.\\
\end{array}
\right.
\end{equation}

\end{definition}

\begin{definition}\label{def2}
Let $T_t^\theta:\mc S_\theta\to\mc S_\theta$ be the semigroup associated to \eqref{heat_eq} with the corresponding boundary conditions for the case $\alpha=\beta=0$. That is, given $f\in \mc S_\theta$, by $T_t^\theta f$ we mean the solution of the homogeneous version of \eqref{heat_eq} with initial condition $f$.
\end{definition}

\begin{definition}
Let $\mc S'_{\theta}$ be the topological dual of $\mc S_{\theta}$ with respect to the topology generated by the seminorms 
 \begin{equation}\label{semi-norm}
\|f\|_{k}=\sup_{u\in[0,1]}|\partial_u^kf(u)|\,,
\end{equation}
where $k\in\mathbb{N}\cup \{0\}$.
In other words, $\mc S'_\theta$ consists of all linear functionals \break $f:\mc S_\theta\to \bb R$ which are continuous with respect to all the seminorms $\Vert \cdot \Vert_k$.
\end{definition}

Let  $\mc D([0,T],\mc S'_\theta)$ (resp. $\mc C([0,T], \mc S'_\theta)$) be the space of trajectories which are right continuous and with left limits (resp. continuous),  taking values in $\mc S'_\theta$.

The expression for $T_t^\theta$,  $\theta\geq 0$, is presented in the next proposition:
\begin{proposition} Let $\theta\geq 0$.
Suppose that $\rho_0\in L^2[0,1]$. Then 
\begin{equation}\label{eq6semi}
(T_t^\theta \rho_0)(u) \;:=\; \sum_{n=1}^\infty a_n\,e^{-\lambda_n t}\,\Psi_n(u)\,,
\end{equation}
where $\{\Psi_n\}_{n\in \bb N}$ is an orthonormal basis of  $L^2[0,1]$ constituted by  eigenfunctions  of the associated  Regular Sturm-Liouville Problem (concerning the operator $\Delta_\theta$) and $a_n$ are the Fourier coefficients of $\rho_0$ in the basis $\{\Psi_n\}_{n\in \bb N}$. 

\begin{itemize}
\item For $\theta <1$, the corresponding orthonormal basis of  $L^2[0,1]$ is 
\[
\begin{cases}
\Psi_n(u)\;=\; \sqrt{2}\sin (n\pi u)\,,\; & \text{ for }n\in \bb N\,,\\
\Psi_0(u)\;\equiv\; 1\,.
\end{cases}
\]
The eigenvalues of the associated  Regular Sturm-Liouville Problem (concerning the operator $\Delta_\theta$) are given by $\lambda_n= n^2\pi^2$.

 \item For $\theta=1$, the corresponding orthornormal basis of $L^2[0,1]$ is a linear combination of sines and cosines, namely,
 \[\Psi_n(u)\;=\; A_n\sin (\sqrt{\lambda_n} u)+A_n\sqrt{\lambda_n}\cos (\sqrt{\lambda_n} u)\,,\; \text{ for }n\in \bb N\cup \{0\}\,,
\]
where $A_n$ is a normalizing constant. The eigenvalues $\lambda_n$ do not have an explicit formula, but it can  verified that $\lambda_n\sim n^2\pi^2$.
  \item For $\theta >1$, the corresponding orthonormal basis of  $L^2[0,1]$ is 
  \[
\begin{cases}
\Psi_n(u)\;=\; \sqrt{2}\cos (n\pi u)\,,\; & \text{ for }n\in \bb N\,,\\
\Psi_0(u)\;\equiv\; 1\,.
\end{cases}
\]
\end{itemize}
The eigenvalues of the associated  Regular Sturm-Liouville Problem (concerning the operator $\Delta_\theta$) are given by $\lambda_n= n^2\pi^2$.

\end{proposition}

\begin{proof}
 For $\theta=1$ the expression for $T_t^\theta$ has been obtained in \cite{fgn4}.
 For the case $\theta\neq 1$, 
 as in \cite{fgn4}, we state the associated \textit{Regular Sturm-Liouville Problem} (for details on this subject we refer to \cite{birkhoff89}, for instance):
 \begin{equation*}
  \mbox{For }~\theta<1:\quad
\begin{cases}&\Psi''(u)+\lambda \Psi(u) \;=\; 0\,, \quad u\in(0,1)\,,    \\
 & \Psi(0)\;=\;0\,, \quad
  \Psi(1)\;=\;0\,;
 \end{cases}
 \end{equation*}
  \begin{equation*}
  \mbox{For }~\theta\geq 1:\quad
\begin{cases}&\Psi''(u)+\lambda \Psi(u) \;=\; 0\,, \quad u\in(0,1)\,,    \\
 & \Psi'(0)\;=\;0\,, \quad
  \Psi'(1)\;=\;0\,.
 \end{cases}
 \end{equation*}
The solution of each one of the problems above (the eigenvalues $\lambda_n$ and the eigenfunctions $\Psi_n$) can be found in  Chapter 10 of \cite{boyce}.
\end{proof}
As a consequence,  the series \eqref{eq6semi} converges exponentially fast, implying  that $(T_t^\theta\rho_0)(u)$ is smooth in space and time for any $t>0$. This observation implies a property of $T_t^\theta:\mc S_\theta\to\mc S_\theta$ stated in the next corollary. 

\begin{corollary}\label{cor42}
If $f\in \mc S_\theta$, then for any $t>0$,  $ T_t^\theta f \in \mc S_\theta$ and  $\Delta_\theta T_t^\theta f\in \mc S_\theta$.
\end{corollary}

We observe that the previous result is needed  in  the proof of uniqueness of the corresponding Ornstein-Uhlenbeck process (which is defined in the next section). Its proof is a consequence  of the formula \eqref{eq6semi}, see \cite{fgn4} for more details.

\subsection{Ornstein-Uhlenbeck process}
Fix  $\rho\in (0,1)$.  Based on \cite{HS, kl}, we give here a characterization of the generalized Ornstein-Uhlenbeck process which is a  solution of
\begin{equation}\label{OU}
d\mathcal{Y}_t=\Delta_\theta \mathcal{Y}_tdt+\sqrt{2\chi(\rho)t}\,\nabla_\theta \, d\mc{W}_t\,,
\end{equation}
 where $\mc{W}_t$ is a space-time white noise of unit variance and $\chi(\rho)=\int (\eta(x)-\rho)^2d\nu^n_\rho=\rho(1-\rho)$, in terms of a martingale problem. We will see below that this process governs the equilibrium fluctuations of the density of particles of our model.  In spite of having a dependence of  $\mathcal{Y}_t$ on $\theta$,  we do not index on it  to not overload notation.  Denote by $\mathbf Q_\rho^{\theta}$ the distribution of 
   $\mathcal{Y}_\cdot$ and $\mathbf E_\rho^{\theta}$ the expectation with respect to $\mathbf Q_\rho^{\theta}$.

 Define the inner product between the  functions $f,g:[0,1]\to\bb R$ by
\begin{equation*}
\<f,\,g\>_{L^{2,\theta}_{\rho}}=2\chi(\rho)\Bigg[\int_{0}^1f(u)\,g(u) \,du+\Big(f(0)g(0)+ f(1)g(1)\Big)\textbf{1}_{\theta=1}\Bigg], 
\end{equation*}
where $\textbf{1}_{\cdot}$ is the indicator function.
  Then,  
$L^{2,\theta}_{\rho}([0,1])$ is the space of functions $f:[0,1]\rightarrow{\mathbb{R}}$ with $\|f\|_{L^{2,\theta}_{\rho}}<\infty$, where
\begin{equation}\label{norma}
\|f\|_{L^{2,\theta}_{\rho}}^2=\<f,\,f\>_{L^{2,\theta}_{\rho}}.
\end{equation}

\begin{proposition}\label{pp1}
There exists an unique random element $\mc Y$ taking values in the space $\mc C([0,T],\mathcal{S}'_{\theta})$ such
that:
\begin{itemize}
\item[i)] For every function $f \in \mathcal{S}_{\theta}$, $\mc M_t(f)$ and $\mc N_t(f)$ given by
\begin{equation}\label{lf1}
\begin{split}
&\mc M_t(f)= \mc Y_t(f) -\mc Y_0(f) -  \int_0^t \mc Y_s(\Delta_\theta f)ds\,,\\
&\mc N_t(f)=\big(\mc M_t(f)\big)^2 - 2\chi(\rho) \; t\,\|\nabla_\theta f\|_{L^{2,\theta}_{\rho}}^2
\end{split}
\end{equation}
are $\mc F_t$-martingales, where for each $t\in{[0,T]}$, $\mc F_t:=\sigma(\mc Y_s(f); s\leq t,  f \in \mathcal{S}_{\theta})$.

\item[ii)] $\mc Y_0$ is a Gaussian field of mean zero and covariance given on $f,g\in{\mathcal{S}_{\theta}}$ by
\begin{equation}\label{eq:covar1}
\mathbf{E}_\rho^\theta\big[ \mc Y_0(f) \mc Y_0(g)\big] =  \<f,\,g\>_{L^{2,\theta}_{\rho}}
\end{equation}
\end{itemize}
Moreover, for each $f\in\mc S_\theta$, the stochastic process $\{\Y_t(f)\,;\,t\geq 0\}$ is  gaussian, being the
distribution of $\Y_t(f)$  conditionally to
$\mc F_s$, for $s<t$, normal  of mean $\Y_s(T_{t-s}^\theta f)$ and variance $\int_0^{t-s}\Vert \nabla_\theta T_{r}^\theta
f\Vert^2_{L^{2,\theta}_{\rho}}\,dr$, where $T_t^\theta $ was given in Definition \ref{def2}.
\end{proposition}
 The  random element $\mc Y_\cdot$ is called the generalized Ornstein-Uhlenbeck process of
characteristics $\Delta_\theta$ and $\nabla_\theta$.  From the second equation in \eqref{lf1} and L\'evy's Theorem on the martingale characterization of Brownian motion, for each $f\in\mathcal S_\theta$,  the process
\begin{equation}\label{Bmotion}
 \mc M_t(f)\Big(2\chi(\rho)t\|\nabla_\theta f\|^2_{L^{2,\theta}_{\rho}}\Big)^{-1/2}
\end{equation}
is a standard Brownian motion. Therefore, in view of  Proposition \ref{pp1}, it makes sense to say that $\Y_\cdot$ is
the formal solution of \eqref{OU}.

\subsection{The density fluctuation field}

We define the density fluctuation field $\mc Y^n_\cdot$ as time-trajectory of the linear functional acting on functions $f\in \mathcal S_\theta$ as
\begin{equation}\label{density field}
\mc Y^n_t(f)\;=\;\frac{1}{\sqrt{n}}\sum_{x=1}^{n-1}f\Big(\frac{x}{n}\Big)\Big(\eta_{tn^2}(x)-\rho^n_t(x)\Big),\quad \mbox{ for all }t\geq 0,
\end{equation}
where $\rho^n_t$ was defined in \eqref{mean}. Our results are given for the case $\alpha=\beta=\rho$ and for $\mu_n$ being equal to $\nu^n_\rho$,  that is, the Bernoulli product measure with parameter $\rho\in (0,1)$, so that $\rho^n_t(x)=\rho,$ for all $x\in \Sigma_n$ and $t\geq 0$. Let $\mc Q_\rho^{\theta,n}$ be the probability measure on $\mc D([0,T],\mc S'_{\theta})$  induced by the density fluctuation field $\Y^n_{\cdot}$.
 We denote ${E}_\rho^{\theta,n}$ the expectation with respect to $\mc Q_\rho^{\theta,n}$. Moreover, 
since  we will  consider  only the initial measure $\mu_n$ as $\nu_\rho^n$, we will simplify the notations $\bb P_{\nu_\rho^n}^{\theta,n}$ and $\bb E_{\nu_\rho^n}^{\theta,n}$ as $\bb P_{\rho}^{\theta,n}$ and $\bb E_{\rho}^{\theta,n}$, respectively. 

Our main result is the following theorem.

\begin{theorem}[Ornstein-Uhlenbeck limit]\label{OU_limit}
\quad

For $\alpha=\beta=\rho\in(0,1)$, if we take the initial measure to be $\nu_\rho^n$, namely, the Bernoulli product measure with parameter $\rho$,
then, the sequence $\{\mc Q_\rho^{\theta,n}\}_{ n\in \mathbb{N}}$ converges, as $n\to\infty$, to a generalized Ornstein-Uhlenbeck (O.U.) process, which is the formal solution of  equation \eqref{OU}. As a consequence, the variance of the limit field $\mathcal{Y}_t$ is given on $f\in{\mc S_\theta}$ by
\begin{equation}\label{covariance_local_gibbs}
\mathbf{E}_\rho^\theta\,[\mathcal{Y}_t(f)\mathcal{Y}_s(f)]\;=\;\chi(\rho)\int_0^1 \,(f(u))^2\,du+\int_0^s\| T_{t-r}^\theta f\|^2_{L^{2,\theta}_{\rho}}dr\,,
\end{equation} 
where $\| \cdot\|^2_{L^{2,\theta}_{\rho}}$ was defined in \eqref{norma}.

\end{theorem}

\section{Proof of Theorem \ref{OU_limit}}\label{sec_proof}

\subsection{Characterization of limit points}

Fix a test function $f$. By Dynkin's formula, we have that
\begin{equation}\label{martingaleM}
M^n_t(f)=\mc Y_t^n(f)-\mc Y_0(f)-\int_{0}^t(\partial_s+n^2\mc L_n^\theta)\mc Y_s^n(f)ds,
\end{equation}
\begin{equation}\label{martingaleN}
N^n_t(f)=(M_t^n(f))^2-\int_{0}^tn^2\mc L_n^\theta \mc Y_s^n(f)^2-2\mc Y_s^n(f)n^2\mc L_n^\theta \mc Y_s^n(f)ds
\end{equation}
are martingales with respect to the natural filtration $\mc F_t:=\sigma(\eta_s:\, s\leq t)$. To simplify notation we denote  $\Gamma_s^n(f):=(\partial_s+n^2\mc L_n^\theta)\mc Y_s^n(f)$.
A long but elementary computation shows that
\begin{equation}\label{int part of mart}
\begin{split}
\Gamma_s^n(f)=&\frac{1}{\sqrt{n}}\sum_{x=1}^{n-1}\Delta_n f\Big(\tfrac{x}{n}\Big)(\eta_s(x)-\rho)\\
+& \sqrt n \nabla_n^+f(0)(\eta_s(1)-\rho)-\sqrt n \nabla_n^-f({n})(\eta_s(n-1)-\rho)\\
-&\frac{n^{3/2}}{n^\theta}f\Big(\tfrac{1}{n}\Big)(\eta_s(1)-\rho)-\frac{n^{3/2}}{n^\theta}f\Big(\tfrac{n-1}{n}\Big)(\eta_s(n-1)-\rho).
\end{split}
\end{equation}
Above
\begin{equation*}
\Delta_nf(x):=n^2\Big[f\Big(\tfrac{x+1}{n}\Big)+f\Big(\tfrac{x-1}{n}\Big)-2f\Big(\tfrac{x-1}{n}\Big)\Big]\,,
\end{equation*}
\begin{equation*}
\nabla_n^+f(x):=n\Big[f\Big(\tfrac{x+1}{n}\Big)-f\Big(\tfrac{x}{n}\Big)\Big]
\end{equation*}
and
\begin{equation*}
\nabla_n^-f(x):=n\Big[f\Big(\tfrac{x}{n}\Big)-f\Big(\tfrac{x-1}{n}\Big)\Big]\,.
\end{equation*}

We note that for the choice  $\theta=0$, using the fact  that $f(0)=0=f(1)$, the expression \eqref{int part of mart} reduces to
\begin{equation}\label{gamma_eq}
\Gamma_s^n(f)=\frac{1}{\sqrt{n}}\sum_{x=1}^{n-1}\Delta_n f\Big(\tfrac{x}{n}\Big) (\eta_s(x)-\rho)\,,
\end{equation}
which is $\mc Y_s^n(\Delta_n f).$

Now, we close the equation \eqref{int part of mart} for each regime of $\theta$. The goal is two show that we can rewrite \eqref{int part of mart} as \eqref{gamma_eq} plus a term which vanishes as $n\to\infty$. 

\quad 

$\bullet$ The case $\theta<1$: we note that since $f\in{\mathcal S_\theta}$ we can write $\Gamma_s^n(f)$ as
\begin{equation*}
\mc Y_s^n(\Delta_n f)+ \sqrt n(1-n^{-\theta})\Big\{ \nabla_n^+f(0)(\eta_s(1)-\rho)- \nabla_n^-f(n)(\eta_s(n-1)-\rho)\Big\}.
\end{equation*}

In order to close the equation for the martingale we need to show that:
\begin{equation}\label{rep_1}
\lim_{n\to\infty}\mathbb{E}_{\rho}^{\theta,n}\big[ \Big(\int_{0}^t\sqrt n\Big(\eta_{s}(x)-\rho\Big) \, ds\Big)^2\Big]=0,\quad\mbox{for } x=1,n-1,
\end{equation}
which is a consequence of Lemma \ref{rep_lemma}, see Remark \ref{rem_rep}.

\quad 

$\bullet$ The case $\theta=1$: we can write $\Gamma_s^n(f)$ as
\begin{equation*}
\begin{split}
&\mc Y_s^n(\Delta_n f)+ \sqrt n \Big( \p_uf(0)-f(0)\Big)(\eta_s(1)-\rho)\\+&\sqrt n \Big(\p_u f(1)+f(1)\Big)(\eta_s(n-1)-\rho)+O\Big(\frac{1}{\sqrt n}\Big).
\end{split}
\end{equation*}
Since $f\in{\mathcal S_\theta}$ the last expression equals to $\mc Y_s^n(\Delta_n f).$ 

\quad 

$\bullet$ The case $\theta>1$: we can repeat the computations above and since $f\in{\mathcal S_\theta}$, in order to close the equation for the martingale term we need to show that
\begin{equation}\label{rep_2}
\lim_{n\to\infty}\mathbb{E}^{\theta,n}_{\rho}\Big[ \Big(\int_{0}^t\frac{ n^{3/2}}{n^\theta}(\eta_s(x)-\rho) \, ds\Big)^2\Big]=0, \quad\mbox{for } x=1,n-1\,,
\end{equation}
which is a consequence of Lemma \ref{rep_lemma}, see Remark \ref{rem_rep}.

From the previous observations, for each regime of $\theta$ we can rewrite \eqref{int part of mart} as \eqref{gamma_eq} plus a negligible term.

\begin{lemma}\label{lemma111}
For all $\theta\geq 0$, $t>0$ and $f\in\mc S_\theta$ it holds that
\begin{equation*}
\lim_{n\to\infty} \mathbb{E}_{\rho}^{\theta,n}[|M_{t}^{n}(f)|^{2}]= t
\|\nabla_\theta f\|_{L^{2,\theta}_{\rho}},
\end{equation*}
where the norm above was defined in \eqref{norma}.
\end{lemma}

\begin{proof}

A simple computation shows that the integral part of the martingale $N_t^n(f)$ can be written as 
\begin{equation*}\label{quad var}
\begin{split}
n^2\mc L_n^\theta \mc Y_s^n(f)^2- 2\mc Y_s^n(f)n^2& \mc L_n^\theta \mc Y_s^n(f)=\frac{1}{n}\sum_{x=1}^{n-2}\Big(\nabla^+_n f\Big(\tfrac{x}{n}\Big)\Big)^2 \Big(\eta_s(x)-\eta_s(x+1)\Big)^2\\
+& \frac{n}{n^\theta}\, \Big(f\big(\pfrac{1}{n}\big)\Big)^2\Big(\rho-2\rho \eta_s(1)+\eta_s(1)\Big)\\
+&\frac{n}{n^\theta}\, \Big(f\big(\pfrac{n-1}{n}\big)\Big)^2\Big(\rho-2\rho \eta_s(n-1)+\eta_s(n-1)\Big),
\end{split}
\end{equation*}
from where we get that 
\begin{equation}\label{exp^2}
\begin{split}
&\mathbb{E}_{\rho}^{\theta,n}\big[|M_{t}^{n}(f)|^{2}\big]\\
&=2\chi( \rho)t\Bigg\{\frac{1}{n}\sum_{x=1}^{n-2}\Big(\nabla^+_n f\Big(\tfrac{x}{n}\Big)\Big)^2 
+\frac{n}{n^\theta}\, \Bigg(\Big(f\big(\pfrac{1}{n}\big)\Big)^2+
 \Big(f\big(\pfrac{n-1}{n}\big)\Big)^2\Bigg)\Bigg\}.
\end{split}
\end{equation}

Let $f\in{\mc{S_\theta}}$. The first term at the right hand side of the previous expression converges to $ 2\chi( \rho)\int_0^1\Big(\nabla_\theta f(u)\Big)^2\,du$, for all $\theta \geq 0$. The second term at the right and side of last expression has to be  analyze for each case of $\theta$ separately:

\quad 

$\bullet$ The case $\theta<1$: since $f(0)=0=f(1)$, the second term at the right hand side of \eqref{exp^2} can be rewritten as $2\chi( \rho)t$ times
$$
\frac{n}{n^\theta}\, \Bigg(\Big(f\big(\pfrac{1}{n}\big)\Big)^2+
 \Big(f\big(\pfrac{n-1}{n}\big)\Big)^2\Bigg)=\frac{1}{n^{1+\theta}}\, \Bigg(\Big(\nabla_n^+f(0)\Big)^2+
 \Big(\nabla_n^-f(n)\Big)^2\Bigg),
$$
which goes to zero as $n\to\infty$.

\quad 

$\bullet$ The case $\theta=1$: the second term at the right hand side of \eqref{exp^2} converges, as $n\to\infty$, to 
$$2\chi( \rho)\Big(f^2\big(0\big)+
 f^2\big(1\big)\Big).
$$ Recalling that $f(0)=\partial_u f(0)$ and $f(1)=-\partial_u f(1)$,  the proof ends.

\quad 

$\bullet$ The case $\theta>1$: since $f\in{\mc{S_\theta}}$ and $\frac{n}{n^\theta}\to 0$, as $n\to\infty$,  the second term at the right hand side of \eqref{exp^2} converges to zero when $n\to\infty$.

\end{proof}
 We have just proved that the quadratic variation of the martingale converges in mean. In the next Lemma below we state the stronger convergence of the martingales to a Brownian motion. 

\begin{lemma}\label{lemma32}
For $f\in\mc S_\theta$, the sequence of martingales $\{M^n_t(f);t\in [0,T]\}_{n\in \bb N}$ converges in the topology of $\mc D([0,T], \bb R)$, as $n\to\infty$, towards a Brownian motion $\mathcal{W}_t(f)$ of  quadratic variation given by 
$ t\|\nabla_\theta f\|_{L^{2,\theta}_{\rho}}$
where $\|\cdot\|_{L^{2,\theta}_{\rho}}$ was defined in \eqref{norma}.
\end{lemma}

\begin{proof}
We can repeat here the same proof of  \cite[page 4170]{fgn3}, which is based on   
Lemma \ref{lemma111} and the fact that a limit in distribution of a uniformly
integrable sequence of martingales is a martingale. We leave the details to the interested reader. 
\end{proof}  

\subsection{Convergence at initial time}

\begin{proposition} \label{convergence at time zero}

The sequence  $\{\mc Y^n_0\}_{n\in\mathbb{N}}$ converges in distribution to $\mc Y_0$,
where $\mc Y_0$
 is a Gaussian field with mean zero and covariance given by \eqref{eq:covar1}.
\end{proposition}
\begin{proof}
We first claim that, for every $f\in\mathcal{S}_\theta$ and every $t>{0}$,
\begin{equation*}\label{eq23}
\lim_{n\rightarrow{{+\infty}}}\log 
\,{E}_\rho^{\theta,n}\Big[\exp\{i\lambda{\mathcal{Y}^n_{0}(f)}\}\Big]=-\frac{\lambda^2}{2}
\chi(\rho)\int_{0}^1 f^2(u)\,du\,.
\end{equation*}
Since $\nu^n_\rho$ is a Bernoulli product measure,
\begin{equation*}
\begin{split}
\log\,{E}_\rho^{\theta,n}[\exp\{i\lambda\mathcal{Y}^n_{0}(f)\}]&=
\log\bb{E}_\rho^{\theta}\Big[\exp\Big\{\frac{i\lambda}{\sqrt{n}}\sum
_{x	\in{\Sigma_n}}\; (\eta_0(x)-\rho)\,f\Big(\frac{x}{n}\Big)\Big\}\Big]\\
&=\sum_{x\in{\Sigma_n}}\log\bb {E}_\rho^{\theta}\Big[\exp\Big\{\frac{i\lambda}{\sqrt{n}}\;
(\eta_0(x)-\rho)\,f\Big(\frac{x}{n}\Big)\Big\}\Big]\,.
\end{split}
\end{equation*}
Since $f$ is smooth and using Taylor's expansion, the right hand side of last expression is equal to
\begin{equation*}
-\frac{\lambda^2}{2n}\sum_{x\in{\Sigma_n }}f^{2}\Big(\frac{x}{n}\Big)\chi(\rho)+O\Big(\frac{1}{\sqrt n}\Big)\,.
\end{equation*}
Taking the limit as $n\rightarrow{+\infty}$ and using the continuity of $f$, the proof of the claim ends. Replacing $f$ by a
linear combination of functions and recalling the Cr\'amer-Wold device, the proof finishes.
\end{proof}

\section{Tightness}\label{s6}
Now we prove that the sequence of processes $\{\mc Y_t^n; t \in [0,T]\}_{n \in \bb N}$ is tight. Recall that we have defined the density fluctuation field on test functions $f\in\mc S_\theta$. Since we want to use  Mitoma's criterium   \cite{Mitoma} for tightness,  we need the following property from the space $\mathcal S_\theta$.
\begin{proposition}\label{frechet}
 The space $\mc S_\theta$ endowed with the semi-norms given in \eqref{semi-norm}
is a Fr\'echet space.
\end{proposition}
\begin{proof}
The definition of a Fr\'echet space can be found, for instance, in \cite{reedsimon}. Since $C^{\infty}([0,1])$   endowed with the semi-norms \eqref{semi-norm}
is a Fr\'echet space, and a closed subspace of a Fr\'echet space is also a Fr\'echet space, it is enough to show that   $\mc S_\theta$ is a closed subspace of $C^\infty([0,1])$, which   is a consequence of the fact that uniform convergence implies point-wise convergence.
\end{proof}
As a consequence of Mitoma's
criterium  \cite{Mitoma} and Proposition \ref{frechet}, the proof of tightness of the $\mc S'_\theta$ valued processes  $\{\mc Y_t^n; t \in [0,T]\}_{n \in \bb N}$ follows from tightness of the sequence of real-valued processes $\{\mc Y_t^n(f); t \in [0,T]\}_{n \in \bb N}$,
for $f\in{\mc S_\theta}$.

\begin{proposition}[Mitoma's criterium,  \cite{Mitoma}]\quad
A sequence  of processes $\{x_t;t \in [0,T]\}_{n \in \bb N}$  in $\mc D([0,T],\mc {S_\theta}')$ is tight with respect to the
Skorohod topology if, and only if, the sequence $\{x_t(f);t \in [0,T]\}_{n \in \bb N}$ of real-valued processes is tight with
respect to the Skorohod topology of $\mc D([0,T], \bb R)$, for any $f \in \mc {S_\theta}$.
\end{proposition}
Now, to show tightness of the real-valued process we use the Aldous' criterium:
\begin{proposition}
 A sequence $\{x_t; t\in [0,T]\}_{n \in \bb N}$ of real-valued processes is tight with respect to the Skorohod topology of $\mc
D([0,T],\bb R)$ if:
\begin{itemize}
\item[i)]
$\displaystyle\lim_{A\rightarrow{+\infty}}\;\limsup_{n\rightarrow{+\infty}}\;\mathbb{P}_{\mu_n}\Big(\sup_{0\leq{t}\leq{T}}|x_{t
} |>A\Big)\;=\;0\,,$

\item[ii)] for any $\varepsilon >0\,,$
 $\displaystyle\lim_{\delta \to 0} \;\limsup_{n \to {+\infty}} \;\sup_{\lambda \leq \delta} \;\sup_{\tau \in \mc T_T}\;
\mathbb{P}_{\mu_n}(|
x_{\tau+\lambda}- x_{\tau}| >\varepsilon)\; =\;0\,,$
\end{itemize}
where $\mc T_T$ is the set of stopping times bounded by $T$.
\end{proposition}
Fix $f\in{\mc S_\theta}$. By \eqref{martingaleM}, it is enough to prove tightness of $\{\mc Y_0^n(f)\}_{n \in
\bb N}$, $\{ \int_{0}^t\Gamma_s^n(f)\, ds; t \in [0,T]\}_{n \in \bb N}$, and $\{\mc M_t^n(f
); t \in [0,T]\}_{n \in \bb N}$. 
\subsection{Tightness at the initial time}
This follows from Proposition \ref{convergence at time zero}.

\subsection{Tightness of the martingales}
By Lemma \ref{lemma32}, the sequence of martingales converges. In particular, it is tight.

\subsection{Tightness of the integral terms}

The first claim of Aldous' criterium can be easily checked for the integral term $\int_{0}^t\Gamma_s^n(f)\, ds$, where the expression for $\Gamma_s^n(f)$ can be found in \eqref{int part of mart}. Let $f\in{\mc{S_\theta}}$.
\begin{itemize}
\item The case  $\theta<1$: by Young's inequality and Cauchy-Schwarz's inequality we have that
\begin{equation*}
\begin{split}
{E}_{\rho}^{\theta,n}\Big[\sup_{t\leq {T}}\Big(&\int_{0}^t\Gamma_s^n(f)\, ds\Big)^2\Big]\\&\leq CT \int_{0}^T \mathbb{E}^{\theta,n}_{\rho}\Big[\Big(\frac{1}{\sqrt{n}}\sum_{x=1}^{n-1}\Delta_n f(\tfrac{x}{n})(\eta_{sn^2}(x)-\rho)\Big)^2\Big]\, ds\\
&+C\,(\nabla_n^+f(0))^2\,T \int_{0}^T \mathbb{E}^{\theta,n}_{\rho}\Big[\Big(\sqrt{n}(\eta_{sn^2}(1)-\rho)\Big)^2\Big]\, ds\\
&+C\,(\nabla_n^-f({1}))^2\,T \int_{0}^T \mathbb{E}^{\theta,n}_{\rho}\Big[\Big(\sqrt{n}(\eta_{sn^2}(n-1)-\rho)\Big)^2\Big]\, ds.
\end{split}
\end{equation*}
Since $f\in\mc S_\theta$ and by  \eqref{rep_1}, the second and third terms at the right hand side of the previous expression go to zero, as $n\to\infty$. Then there exists $C>0$ such that these two terms are bounded from above by $CT$.
The first term at the right hand side of last expression is bounded from above by $T^2$ times
\begin{equation}\label{estimate111}
\frac{1}{{n}}\sum_{x=1}^{n-1}\big(\Delta_n f(\tfrac{x}{n})\Big)^2\chi(\rho)
\,.
\end{equation}
Now, since $f\in\mc S_\theta$ last expression is bounded from above by a constant. Now we need to check the second claim of Aldous' criterium. For that purpose,
fix a stopping time $\tau \in \mc T_T$. By Chebychev's inequality together  with \eqref{estimate111}, we get that
 \begin{equation*}
{P}_{\rho}^{\theta,n}\Big(\Big|  \int_{\tau}^{\tau+\lambda}\!\!\Gamma^n_s(f)\, ds\;\Big| >\varepsilon\Big)
	\leq \frac{1}{\varepsilon^2} {E}_{\rho}^{\theta,n}\Big[ \Big(  \int_{\tau}^{\tau+\lambda}\!\!\Gamma^n_s(f)\; ds \;\Big)^2\Big]
	\leq \frac{\delta C}{\varepsilon^2}\,,
\end{equation*}
which vanishes as $\delta\rightarrow{0}$.

\quad 

\item The case $\theta=1$: we note that it was treated in  \cite{fgn4}.

\quad 

\item  The case $\theta>1$: as in the case $\theta<1$, we have that 
\begin{equation*}
\begin{split}
{E}_{\rho}^{\theta,n}\Big[\sup_{t\leq {T}}\Big(&\int_{0}^t\Gamma_s^n(f)\, ds\Big)^2\Big]\\&\leq CT \int_{0}^T \mathbb{E}^{\theta,n}_{\rho}\Big[\Big(\frac{1}{\sqrt{n}}\sum_{x=1}^{n-1}\Delta_n f(\tfrac{x}{n})(\eta_{sn^2}(x)-\rho)\Big)^2\Big]\, ds\\
&+C\,f^2\Big(\tfrac{1}{n}\Big)\,T \int_{0}^T \mathbb{E}^{\theta,n}_{\rho}\Big[\Big(\frac{n^{3/2}}{n^\theta}(\eta_{sn^2}(1)-\rho)\Big)^2\Big]\, ds\\
&+C\,f^2\Big(\tfrac{n-1}{n}\Big)\,T \int_{0}^T \mathbb{E}^{\theta,n}_{\rho}\Big[\Big(\frac{n^{3/2}}{n^\theta}(\eta_{sn^2}(n-1)-\rho)\Big)^2\Big]\, ds\,,
\end{split}
\end{equation*}
plus a term of order $\frac{1}{\sqrt{n}}$. To bound the first term at the right hand side of the previous  inequality  we repeat  the same computations  as in the case $\theta<1$. In order to bound the second and the  third terms  at the right hand side of  the previous  inequality, we use \eqref{rep_2} and  the proof follows  as in the case $\theta<1$.

\end{itemize}

\section{Replacement Lemma} \label{sec_rl}

This section is devoted to estimate the expectations \eqref{rep_1} and \eqref{rep_2}. In order to do this we start introducing some notations.
Let $\mu$ be an initial  measure. For $x=0,1,\dots,n-1$, define
\begin{equation*}
I_{x,x+1}(f,\mu)\;:=\; \int r_{x,x+1}(\eta)\left( {f(\sigma^{x,x+1}\eta)}-{f(\eta)}\right)^{2} d\mu\,,
\end{equation*}
where  $\sigma^{x,x+1}\eta$ was defined in \eqref{sigma}, for  $x=1,...,n-2$,
$\sigma^{0,1}\eta:=\eta^{1}$, $\sigma^{n-1,n}\eta:=\eta^{n-1}$ (the configurations $\eta^{1}$ and $\eta^{n-1}$ were defined in \eqref{split}), and the rates are given by 
\begin{equation*}
\begin{split}
r_{0,1}(\eta)\;& :=\;r_{\alpha}(\eta):=\frac{\alpha}{n^\theta}(1-\eta(1)) + \frac{1-\alpha}{n^\theta}\eta(1)\,,\\
r_{n-1,n}(\eta)&\;:=\;r_{\beta}(\eta):=\frac{\beta}{n^\theta}(1-\eta(n-1)) +\frac{1-\beta}{n^\theta}\eta(n-1)\,,\\
r_{x,x+1}(\eta)&\;:=\;1,\; \text{ if } x=1,\cdots,n-2\,.
\end{split}
\end{equation*}
Define the quantity:
\begin{equation}\label{Dn}
\mathcal D_{n}(f,\mu)\;:=\sum_{x=0}^{n-1}I_{x,x+1}(f,\mu)\;= \sum_{x=0}^{n-1}\int r_{x,x+1}(\eta)\left( {f(\sigma^{x,x+1}\eta)}-{f(\eta)}\right)^{2} d\mu\,.
 \end{equation}
 
The Dirichlet form is defined by  $\langle -\mathcal L_{n}^\theta{f},{f} \rangle_{\mu} $, where we can rewrite for short the infinitesimal generator as
\begin{eqnarray*}
\mathcal L_{n}^\theta f(\eta)\;:=\;\sum_{x=0}^{n-1}L_{x,x+1}f(\eta)\;:=\;\sum_{x=0}^{n-1}r_{x,x+1}(\eta)(f(\sigma^{x,x+1}\eta)-f(\eta))\,.\end{eqnarray*}

 Now, we recall that we consider the case  $\alpha=\beta=\rho\in{(0,1)}$, so that the measure $\nu_\rho^n$ (the Bernoulli product measure) is invariant  for this process and it satisfies
  \begin{equation}\label{theta_eta_2}r_{x,x+1}(\eta)\,\nu_\rho^n(\eta)=r_{x,x+1}(\sigma^{x,x+1}\eta)\,
\nu_\rho^n(\sigma^{x,x+1}\eta)\,, \end{equation} for all $x\in\{0,1,\dots,n-1\}$.
Let us check this equality in the case $x=0$, the case $x=n-1$ is similar and the others  are also very simple to check. Note that
  \begin{equation}\label{theta_eta_1}r_{0,1}(\sigma^{0,1}\eta)\,
\frac{\nu_\rho^n(\sigma^{0,1}\eta)}{\nu_\rho^n(\eta)}=
\Big[\frac{\rho}{n^\theta}(1-\eta^1(1)) + \frac{1-\rho}{n^\theta}\eta^1(1)\Big]\frac{\nu_\rho^n(\eta^1)}{\nu_\rho^n(\eta)}. \end{equation} Since
  \begin{equation}\label{theta_eta}
 \frac{\nu_\rho^n(\eta^1)}{\nu_\rho^n(\eta)}=\textbf{1}_{\eta(1)=1}\frac{1-\rho}{\rho}+\textbf{1}_{\eta(1)=0}\frac{\rho}{1-\rho}\,,
 \end{equation}
 then  \eqref{theta_eta_1} becomes
$$\textbf{1}_{\eta(1)=1}\;\Big[\frac{\rho}{n^\theta}\Big]\;\frac{1-\rho}{\rho}\; + \textbf{1}_{\eta(1)=0}\;\Big[\frac{1-\rho}{n^\theta}\Big]\;\frac{\rho}{1-\rho}\;=r_{0,1}(\eta)\,.$$

 Thus, using \eqref{theta_eta_2}, we get
 \begin{equation}\label{eee}\langle -\mathcal L_{n}^\theta{f},{f} \rangle_{\nu_\rho^n} =\frac{1}{2}\mathcal D_{n}(f,\nu_\rho^n)\,.\end{equation}


\begin{lemma}[Replacement Lemma]\label{rep_lemma}
Let  $x=1,n-1$ and $t>0$ fixed. Then
\begin{equation*}
\bb E_{\rho}^{\theta,n}\Big[\Big(\int_0^t c_n\big(\eta_s(x)-\rho\big)\, ds\Big)^2\Big]\;\leq\;  C\frac {c_n^2n^\theta}{n^2}.
\end{equation*}

\end{lemma}
\begin{remark}\label{rem_rep}
Recall that for $\theta<1$ we have in \eqref{rep_1} $c_n=\sqrt n$, so that the error above becomes $n^\theta/n$, which vanishes as $n\to\infty$. 
Recall that for $\theta>1$ we have in \eqref{rep_2} $c_n=n^{3/2}/n^\theta$, so that the error above becomes $n/n^\theta$, which vanishes as $n\to\infty$. 
\end{remark}
\begin{proof}
The proof follows by a classical argument combining both the Kipnis-Varadhan's inequality (see \cite[page 33, Lemma 6.1]{kl}) with Young's inequality. For that purpose let $x=1$ (the other case is completely analogous) and  note that the expectation in the statement of the lemma can be bounded from above by
\begin{equation}\label{kv}
 \sup_{f\in L^2_{\nu^n_\rho}}\Big\{ \int c_n(\eta(1)-\rho)f(\eta)\,d\nu_\rho^n+ {n^2}\langle \mathcal  L_{n}^\theta{f},{f} \rangle_{\nu_\rho^n}\Big\},
\end{equation}
where $L^2_{\nu^n_\rho}$ is the space of functions $f$ such that $\int f^2(\eta)\,d\nu_\rho^n<+\infty$.
We start by writing the integral    $\int c_n(\eta(1)-\rho)f(\eta)d\nu_\rho^n $ as twice its half and in one of the terms we make the exchange $\eta\to\eta^1$ to have 
\begin{equation*}
 \frac {1}{2}\int c_n(\eta(1)-\rho)f(\eta)\,d\nu^n_\rho +  \frac {1}{2}\int c_n(1-\eta(1)-\rho)f(\eta^1) \frac{\nu_\rho^n(\eta^1)}{\nu_\rho^n(\eta)}\,d\nu^n_\rho\,,
\end{equation*}
see \eqref{theta_eta} to get the expression of $ \frac{\nu_\rho^n(\eta^1)}{\nu_\rho^n(\eta)}$.
A simple computation shows that the integral at the right hand side of last expression is equal to 
\begin{equation*}
- \frac {1}{2}\int c_n(\eta(1)-\rho)f(\eta^1)\,d\nu^n_\rho\,,
\end{equation*}
so that the display above is equal to 
\begin{equation*}
 \frac {1}{2}\int c_n(\eta(1)-\rho)(f(\eta)-f(\eta^1))\,d\nu^n_\rho. 
\end{equation*}
By Young's inequality we can bound the previous expression by
\begin{equation*}
B\int c_n^2(\eta(1)-\rho)^2d\nu^n_\rho+
 \frac {1}{4B}\int ({f(\eta)}-{f(\eta^1)})^2\,d\nu^n_\rho\,.
\end{equation*}
Now, remember the notation $\eta^1=\sigma^{0,1}\eta$ and multiply and divide by $r_{0,1}(\eta)$ the integrand function  inside the second integral above. We can do it, because that there exists $\tilde C_\rho$ such that  $\frac{\tilde C_\rho}{n^\theta}\leq r_{0,1}(\eta)\leq \frac{C_\rho}{n^\theta}$. Then we can bound the previous expression from above by
\begin{equation*}
\begin{split}
& B\int c_n^2(\eta(1)-\rho)^2\,d\nu^n_\rho+
 \frac {n^\theta}{4B\tilde C_\rho}\int r_{0,1}(\eta)\,({f(\sigma^{0,1}\eta)}-{f(\eta)})^2\,d\nu^n_\rho\,.\end{split}
\end{equation*}
Using \eqref{Dn} the second integral in the last expression is bounded from above by $\mathcal D_{n}({f},\nu_\rho^n)$.  Recalling \eqref{eee},  we get
\begin{equation*}
\begin{split}
\int c_n(\eta(1)-\rho)f(\eta)\,d\nu_\rho^n\leq
&\,B\,c_n^2\int (\eta(1)-\rho)^2\,d\nu^n_\rho+
 \frac {n^\theta}{2B\tilde C_\rho}\langle -\mathcal L^\theta_{n}{f},{f} \rangle_{\nu_\rho^n}\,.\end{split}
\end{equation*}
Putting this inequality  in \eqref{kv}  and 
choosing $B=n^{\theta-2}/2\tilde C_\rho$, the term at the right hand side of lthe last expression cancels with $n^2\langle \mathcal L_{n}^\theta{f},{f} \rangle_{\nu_\rho^n}$.
Therefore, the expectation appearing in the statement of the lemma is bounded from above by
\begin{equation*}
 \frac {c_n^2n^\theta}{2\tilde C_\rho n^2}\int (\eta(1)-\rho)^2\,d\nu^n_\rho\,.
\end{equation*}
Since  $\eta$ is bounded  the proof ends. 
 
\end{proof}

\section*{Acknowledgements}~~~
A. N. was supported through a grant ``L'OR\' EAL - ABC - UNESCO Para Mulheres na Ci\^encia''.
P. G. thanks  FCT/Portugal for support through the project 
UID/MAT/04459/2013. T. F. was supported by FAPESB through the project Jovem Cientista-9922/2015.
\bibliographystyle{amsplain}

\bibliography{bibliography}
\end{document}